\documentclass[a4paper,12pt]{article}
\usepackage{amsmath,amsthm,amssymb,graphicx,subfigure}
\usepackage{bbm} \usepackage{graphicx,verbatim}
\usepackage{color}
\usepackage [english]{babel}
\usepackage [latin1] {inputenc}
\usepackage [T1] {fontenc}
\usepackage{amssymb}
\usepackage{amsfonts} 
\usepackage{dsfont}
\usepackage{amsmath}
\usepackage{stmaryrd}
\usepackage{graphicx} 
\usepackage[colorlinks=true,linkcolor=blue]{hyperref}
\newtheorem{theorem}{Theorem}[section]
\newtheorem{lemma}[theorem]{Lemma}
\newtheorem{proposition}[theorem]{Proposition}

\newtheorem{example}[theorem]{Example}

\date{\today}
\title{On smooth mesoscopic linear statistics of the eigenvalues of random permutation matrices} 
\author{Valentin Bahier, Joseph Najnudel}

\begin{document}
\maketitle

\begin{abstract}
We study the limiting behavior of smooth linear statistics of the spectrum of random permutation matrices in the mesoscopic regime, when the permutation follows one of the Ewens measures on the symmetric group.  If we apply a smooth enough test function $f$ to all the determinations of the eigenangles of the permutations, we get a convergence in distribution when the order of the permutation tends to infinity. Two distinct kinds  of limit appear: if $f(0)\neq 0$, we have a central limit theorem with a logarithmic variance, and if $f(0) = 0$, the convergence holds without normalization and the limit involves a scale-invariant Poisson point process. 
\end{abstract}

\subsection*{Notation}
In the present article, we denote  by $A = \mathcal{O}(B)$ and by $A \ll B$ the fact that there exists a constant $C > 0$ such that $ |A| \leq CB$. 
The expressions $A = \mathcal{O}_{x,y} (B)$ and $A \ll_{x,y} B$ mean that there exists a quantity $C_{x,y} > 0$ depending only on $x$ and $y$, such that $|A| \leq C_{x,y} B$. 
 
\section{Introduction}

The spectrum of random permutation matrices has been studied with much attention in the last few decades. On the one hand, working with matrices gives a different way to understand
some of the classical properties satisfied by random permutations. 
On the other hand, the set of permutation matrices can be seen as a finite subgroup of the orthogonal group or the unitary group, and thus an interesting problem consists in studying  how similar  are the spectral behaviors of random permutations and usual ensembles of random orthogonal or unitary matrices.  
For a random permutation matrix following one of the Ewens measures, the number of eigenvalues lying on a fixed arc of the unit circle has been studied in detail by Wieand  \cite{wieand2000permutation}, and satisfies a central limit theorem when the order $n$ goes to infinity, with a variance growing like $\log n$. This rate of growth is similar to what is obtained for the Circular Unitary Ensemble and random matrices on other compact groups, for which a central limit theorem also occurs, as it can be seen in Costin and Lebowitz \cite{costin1995gaussian}, Soshnikov \cite{Soshnikov1999JSP} and Wieand \cite{Wieand2002unitary}. 
A similar result has recently been proven by Bahier  \cite{bahier2019meso}, on the number of eigenvalues lying on a mesoscopic arc, for  a suitable modification of Ewens distributed  permutation matrices, and the growth of the variance is also the same as for the CUE, i.e. the logarithm of $n$ times the length of the interval. 
Some other results on the distribution of eigenvalues of  matrices constructed from random permutations can be found in papers by  Bahier \cite{bahier2018limiting}, Evans  \cite{evans2002eigenvalues}, 
Najnudel and Nikeghbali  \cite{najnudel2013distribution}, Tsou \cite{tsou2018eigenvalue}, Wieand \cite{wieand2003wreath}.

The analogy between the permutation matrices and the CUE is not as strong when we consider smooth linear statistics of the eigenvalues. 
In this case, if we take a fixed, sufficiently smooth test function, it is known that the fluctuations of the corresponding linear statistics tend to a limiting distribution, without normalization, which is unusual for a limit 
theorem. In the CUE case, the distribution is Gaussian, as seen in Diaconis and Shahshahani \cite{diaconisshahshahani}, Johansson \cite{johansson1997}, Diaconis and Evans \cite{diaconisevans}, and the variance is proportional to the squared $H^{1/2}$ norm of the test function.  In the case of permutation matrices, the limiting distribution is not Gaussian anymore: its shape depends on the test function $f$ and can be explicitly described in terms of 
$f$ and a sequence of independent Poisson random variables. More detail can be found in Manstavicius \cite{manstavicius2005linear},  Ben Arous  and Dang
\cite{ben2015permutation}. 

 In the case of mesoscopic linear statistics, one also has a central limit theorem without normalization in the CUE case (see  \cite{soshnikov2000clt}). The behavior of mesoscopic linear statistics of other random matrix ensembles have also been studied: the Gaussian Unitary Ensemble (see \cite{duits2018dyson}), more general Wigner matrices (see \cite{he2017wigner}, \cite{he2019mixed}) and determinantal processes (see \cite{johansson2018meso}), the Circular Beta Ensemble (see \cite{lambert2019beta}), the thinned CUE, for which a random subset of the eigenvalues has been removed (see \cite{berggren2017thinned}). 
 However, the smooth mesoscopic linear statistics  of permutation matrices have  not been previously studied. The main point of the present article is to show that they also satisfy 
 some limit theorems.  

The precise framework is given as follows. We fix a parameter $\theta > 0$, and we consider a sequence $(\sigma_n)_{n \geq 1}$, $\sigma_n$ following the 
 the Ewens($\theta$) distribution on the symmetric group $\mathfrak{S}_n$, that is to say
\[\forall \sigma \in \mathfrak{S}_n, \ \mathbb{P} (\sigma_n = \sigma ) = \mathbb{P}_\theta^{(n)} (\sigma ) = \frac{\theta^{K(\sigma)}}{\theta (\theta + 1) \cdots (\theta + n-1 )}\]
where $K(\sigma)$ denotes the total number of cycles of $\sigma$ once decomposed as a product of cycles with disjoint supports. Note that the particular case $\theta = 1$ corresponds to the uniform distribution on $\mathfrak{S}_n$.
The permutation matrix $M^{\sigma}$ associated with any element $\sigma$ of $\mathfrak{S}_n$ is defined as follows: for all $1\leq i , j\leq n$, 
\[M^{\sigma}_{i,j} = \left\{ \begin{array}{ll}
1 & \text{ if } i = \sigma (j) \\
0 & \text{otherwise}.
\end{array}  \right.\]
A key relationship between the cycle structure of $\sigma$ and the spectrum of the corresponding permutation matrix $M^{\sigma}$ appears in the expression of the characteristic polynomial of $M^{\sigma}$:
\[\forall x \in \mathbb{R}, \quad \det (I-xM^{\sigma}) = \prod\limits_{j=1}^n (1-x^j)^{a^{\sigma}_j}\]
where $a^{\sigma}_j$ denotes the number of $j$-cycles in the decomposition of $\sigma$ as a product of disjoint cycles.
As a consequence, the cycle structure of $\sigma$ is fully determined by the spectrum of $M^{\sigma}$, counted with multiplicity.

In this paper we are interested in the mesoscopic behavior of smooth linear statistics of the spectrum of $M^{\sigma_n}$ when $n$ goes to infinity. More precisely, we fix a function $f$ from $\mathbb{R}$ to $\mathbb{C}$ which satisfies the following regularity conditions:
\begin{equation}\label{eq:hyp}
\left\{ \begin{array}{l}
  f \in \mathcal{C}^2 (\mathbb{R}) \\
  f^\prime , f^{\prime \prime} \in L^1 (\mathbb{R}) \\
  \exists M>0, \quad \exists \alpha >1, \quad \forall x\in \mathbb{R} , \quad  \vert f(x) \vert \leq \frac{M}{(1+\vert x \vert )^\alpha}.
\end{array}\right.
\end{equation}

Moreover, we fix a sequence $(\delta_n)_{n \geq 1}$ in $\mathbb{R}_+^*$ such that $\delta_n \rightarrow 0$ and $n \delta_n \rightarrow \infty$ when $n \rightarrow \infty$, which means that the corresponding scale is mesoscopic (small but large with respect to the average spacing between the eigenvalues of $M^{\sigma_n}$). 
In this article, we  mainly study the following quantity: 
$$X_{\sigma_n, \delta_n} (f) := \sum_{x \in \mathbb{R}, e^{i x} \in S(\sigma_n)}  m_n (e^{i x}) f\left( \frac{x}{2 \pi \delta_n}\right),$$
where $S(\sigma_n)$ denotes the spectrum of $M^{\sigma_n}$ and $m_n (e^{ix})$ is  the multiplicity of $e^{ix}$ as an eigenvalue of $M^{\sigma_n}$. In other words, we sum the function 
$f$ at the  eigenangles of $M^{\sigma_n}$, divided by $2 \pi \delta_n$ and counted with multiplicity. Notice that all the determinations of the eigenangles are considered here, and the set of $x$ involved in the sum is $2 \pi$-periodic. Notice that the sum giving $X_{\sigma_n, \delta_n} (f)$ is absolutely convergent, because of the assumption we make on the decay of $f$ at infinity. 
We will also consider the version of the linear statistics where we restrict the sum to the determinations of the eigenangles which are in the interval $(-\pi, \pi]$: 
$$X'_{\sigma_n, \delta_n} (f) := \sum_{x \in (\pi, \pi], e^{i x} \in S(\sigma_n)}  m_n (e^{i x}) f\left( \frac{x}{2 \pi \delta_n}\right).$$
In order to state our main theorem, we need to introduce the Fourier transform of $f$, normalized as follows: 
$$ \hat{f} (\lambda) := \int_{\mathbb{R}} f(x) e^{-2 i \pi x \lambda} d x,$$
and the two following functions from $\mathbb{R}_+^*$ to $\mathbb{C}$:
\[\Theta_f : x \mapsto \sum\limits_{k\in \mathbb{Z}} f(kx)\]
and
\[\Xi_f :  x \mapsto \Theta_f(x) - f(0) \mathds{1}_{x > 1} - \frac{1}{x} \hat{f} (0).\]

The series defining $\Theta_f$ is absolutely convergent because of the assumptions \eqref{eq:hyp}. 
Our main result can now be stated as follows:

\begin{theorem}\label{thm:main}
Let $(\delta_n)_{n \geq 1}$ be a positive sequence such that $\delta_n \underset{n \rightarrow \infty}{\longrightarrow} 0$ and $n \delta_n 
\underset{n \rightarrow \infty}{\longrightarrow} \infty$, and let $f$ be a function from $\mathbb{R}$ to $\mathbb{C}$ satisfying the assumptions \eqref{eq:hyp} given above.
\begin{enumerate}
\item[(i)] If $f(0) \neq 0$, then we have the following asymptotics: 
\[\mathbb{E} \left(    X_{\sigma_n, \delta_n} (f)  \right) = n \delta_n \hat{f}(0) - \theta \log (\delta_n ) f(0) + \mathcal{O}_{f,\theta} (1)\]
and
\[\mathrm{Var} \left(  X_{\sigma_n, \delta_n} (f)  \right) = - \theta \log (\delta_n ) f(0)^2 + \mathcal{O}_{f,\theta} (\sqrt{-\log (\delta_n ) }). \]
Moreover, the following central limit theorem holds: 
$$\frac{  X_{\sigma_n, \delta_n} (f) - \mathbb{E} \left(    X_{\sigma_n, \delta_n} (f)  \right) }{ \sqrt{ \mathrm{Var} \left(  X_{\sigma_n, \delta_n} (f)  \right) } } 
\overset{d}{\underset{n \rightarrow \infty}{\longrightarrow}} \mathcal{N}(0,1).$$
\item[(ii)] If $f(0) = 0$, then we have the following convergence in distribution: 
\[ X_{\sigma_n, \delta_n} (f)  - n \delta_n \hat{f}(0) \overset{d}{\underset{n\to\infty}{\longrightarrow}} \sum_{y\in \mathcal{X}}  \Xi_f (y)  \]
where $\mathcal{X}$ is a Poisson point process with intensity $\frac{\theta}{x} \mathrm{d}x$ on $(0,+\infty)$, and where the sum on $\mathcal{X}$ in the right-hand side is a.s. absolutely convergent. 
\item[(iii)] For any $\alpha > 1$ such that \eqref{eq:hyp} is satisfied, the results given in \it{(i)} and \it{(ii)} are still true if we replace $X_{\sigma_n, \delta_n} (f) $
by $X'_{\sigma_n, \delta_n} (f) $, as soon as $\delta_n = o(n^{-1/\alpha})$ when $n \rightarrow \infty$. 
\end{enumerate}
\end{theorem}

\section{Expression of $X_{\sigma_n, \delta_n} (f) $ in terms of $\Theta_f$ and $\Xi_f$}
In the spectrum of $M^{\sigma_n}$, each cycle of length $\ell$ gives eigenangles equal to all multiples of $ 2 \pi/ \ell$. 
The contribution of these eigenangles in the sum $X_{\sigma_n, \delta_n}(f)$ is: 
\[\sum_{k \in \mathbb{Z}} f\left( \frac{k}{\ell \delta_n}\right) =\Theta_{f}\left( \frac{1}{\ell \delta_n}\right).\]
Then, if we denote by $a_{n,j}$ the number of $j$-cycles in the decomposition of $\sigma_n$ as a product of cycles with  disjoint support, we get: 
\begin{align*}
X_{\sigma_n, \delta_n}(f) &= \sum_{\ell = 1}^n a_{n,\ell} \Theta_f \left( \frac{1}{\ell \delta_n} \right) \\
	&= \sum_{\ell = 1}^n a_{n,\ell} \Xi_f \left( \frac{1}{\ell \delta_n} \right) + f(0) \sum_{\ell < \delta_n^{-1}} a_{n, \ell} + \hat{f}(0) \sum_{\ell = 1}^n \ell \delta_n a_{n, \ell}. 
\end{align*}

Since the total number of elements of all cycles is $n$, we have 
\[\sum_{\ell = 1}^n \ell a_{n, \ell} = n,\]
and then 
\begin{equation}
X_{\sigma_n, \delta_n}(f)
   = n \delta_n \hat{f}(0) + \sum_{\ell = 1}^n a_{n, \ell} \Xi_f\left( \frac{1}{\ell \delta_n}\right) + 
 f(0) \sum_{\ell < \delta_n^{-1}} a_{n, \ell}. \label{Xxi}
 \end{equation}
Note that $\hat{f}(0)$ is the integral of $f$, and then the term $n \delta_n \hat{f} (0)$ is what we would obtain with a constant density of eigenangles of $n/2\pi$.

First note that under \eqref{eq:hyp}, the Poisson summation formula applies and gives for all $x>0$, 
\begin{equation}
\Theta_{f} (x)  = \frac{1}{x} \Theta_{\hat{f}} \left(\frac{1}{x} \right).
\end{equation}
We now get the following asymptotic result on $\Theta_{f} $:
\begin{proposition}
Assume \eqref{eq:hyp}. Then, 
\begin{enumerate}
\item[(i)] $\Theta_f$ is continuous on $\mathbb{R}_+^*$ and converges at infinity to $f(0)$ with rate dominated by $\frac{1}{x^\alpha}$ (where $\alpha$ is given by \eqref{eq:hyp}).
\item[(ii)] $\Theta_{\hat{f}}$ is continuous on $\mathbb{R}_+^*$ and converges at infinity to $\hat{f}(0)$ with rate dominated  by $\frac{1}{x^2}$.
\end{enumerate}
\end{proposition}

\begin{proof}
We prove the two items separately. 
\begin{itemize}
\item \underline{Proof of $(i)$}: Since $f$ is assumed to be continuous, the functions $f_k : x \mapsto f(kx)$ are clearly continuous on $(0,+\infty)$ for all $k\in \mathbb{Z}$. Moreover, for all $k\in \mathbb{Z}\setminus \{0\}$ and for all $x$ in any compact set $[A,B] \subset (0,+\infty)$,
\[\vert f_k (x) \vert \leq \frac{M}{(1+ \vert kx \vert^\alpha )} \leq \frac{M}{\vert k \vert^\alpha A^\alpha}, \]
hence $\sum_k f_k$ converges uniformly on compact sets of $(0,+\infty)$. We deduce the continuity of $\Theta_f$. For the convergence of $\Theta_f$ to $f(0)$ at infinity, we only have to notice that for all $x\geq 1$, 
\[\sum_{k\neq 0} \vert f(kx) \vert \leq \frac{M}{x^\alpha} \sum_{k\neq 0} \frac{1}{\vert k \vert^\alpha} \ll_{M, \alpha} \frac{1}{x^\alpha}.\]
\item \underline{Proof of $(ii)$}:
It is clear that the functions $g_k : x \mapsto \hat{f}(kx)$ are continuous on $(0,+\infty)$ for all $k\in \mathbb{Z}$, and from two consecutive integrations by parts, it follows that for all $k\in \mathbb{Z}\setminus \{0\}$ and for all $x$ in any compact set $[A,B] \subset (0,+\infty)$,
\begin{align*}
\vert g_k (x) \vert &= \left\vert \frac{1}{(2 i \pi k x)^2} \int_{-\infty}^{+\infty} f^{\prime \prime} (y) \mathrm{e}^{-2i\pi k x y} \mathrm{d}y \right\vert \\
	&\leq \frac{1}{4\pi^2 k^2 A^2} \int_{-\infty}^{+\infty} \vert f^{\prime \prime} (y) \vert \mathrm{d}y, 
\end{align*}
hence $\sum_k g_k$ converges uniformly on compact sets of $(0,+\infty)$. Note that there is no boundary term in the integration by parts, since by assumption, $f$ goes to zero at infinity, and $f'$ and $f''$ are integrable, which implies that
$f'$ also goes to zero at infinity. 

Now, for all $x\geq 1$,
\[\sum_{k\neq 0} \vert \hat{f}(kx) \vert \leq \frac{1}{x^2} \times \frac{1}{4\pi^2} \int_{-\infty}^{+\infty} \vert f^{\prime \prime} (y) \vert \mathrm{d}y \sum_{k\neq 0} \frac{1}{k^2} \ll \frac{1}{x^2}\]
and the proof is complete.
\end{itemize}
\end{proof}

From the proposition just above, we deduce the following lemma: 

\begin{lemma}\label{lem:xi}
If the function $f$ satisfies  the assumptions \eqref{eq:hyp}, then for all $x \in \mathbb{R}_+^*$, 
$$|\Xi_f (x)| \ll_{f} \min(x, 1/x).$$
In particular, 
$$\int_0^{\infty}  \frac{|\Xi_f (x)|}{x} dx < \infty.$$
Moreover, $\Xi_f (x)$ is continuous at any point of $\mathbb{R}_+^* \backslash \{1\}$, and also at $1$ if $f(0) = 0$. 
\end{lemma}

\begin{proof}
We have, for all $x\in (0,1]$, 
\[\vert \Xi_f (x) \vert = \left\vert \Theta_f (x) - \frac{1}{x} \hat{f}(0) \right\vert = \frac{1}{x} \left\vert \Theta_{\hat{f}} \left( \frac{1}{x}\right) - \hat{f}(0) \right\vert \ll_f \frac{1}{x} \times x^2 = x,\]
and for all $x \in (1, \infty)$, 
\[\vert \Xi_f (x) \vert \leq \left\vert \Theta_f (x) - f(0) \right\vert + \frac{1}{x} \vert \hat{f}(0) \vert  \ll_{f,\alpha} \frac{1}{x^\alpha} + \frac{1}{x} \ll \frac{1}{x}.\]
The continuity of $\Xi_f$ is an immediate consequence of the continuity of $\Theta_f$. 
\end{proof}

\section{The Feller coupling}
In \cite{feller1945}, Feller introduces a construction of a  uniform permutation on the symmetric group, such that the cycle lengths are given by the spacings between successes in 
independent Bernoulli trials. This construction can be extended to general Ewens distributions, and provides a coupling between the cycle counts of a random permutation and a sequence of independent Poisson random variables. 
For the detail of the coupling procedure and many related results,  we refer to \cite{arratia2003logarithmic} and \cite[Section 4]{ben2015permutation}. 
From the Feller coupling, we can deduce the following lemma: 
\begin{lemma}\label{lem:fellercoupling}
For all $n$, one can couple the numbers $a_{n,\ell}$ of $\ell$-cycles in the random permutation $\sigma_n$, with a sequence of independent Poisson variables $W_\ell$ of parameters $\theta/\ell$, in such a way that
\[\mathbb{E} \left( \left( \sum_{\ell \leq n} \vert a_{n,\ell} - W_\ell \vert  \right)^2 \right) \leq C(\theta)\]
where $C (\theta )$ is a constant which does not depend on $n$.
\end{lemma}

\begin{proof}
The Feller coupling gives, for a sequence $(\xi_j)_{j \geq 1}$ of independent Bernoulli random variables, $\xi_j$ with parameter $\theta/(j-1 + \theta)$, $a_{n, \ell}$ equal to the number of $\ell$-spacings between consecutive "1"  in the sequence $(\xi_1, \dots, \xi_n, 1)$, and $W_{\ell}$ equal to the number of $\ell$-spacings between consecutive "1" in the infinite sequence $(\xi_j)_{j \geq 1}$. 
We deduce that $a_{n,\ell} \leq W_{\ell}$, except for at most one value of $\ell$, for which $a_{n,\ell} $ may be equal to $W_{\ell} +1$. 
We then get 
$$\sum_{\ell \leq n } \vert a_{n,\ell} - W_\ell \vert  \leq 2 + \sum_{\ell \leq n } (W_{\ell} -   a_{n,\ell} ). $$
It is then enough to bound the $L^2$ norm of $G_n - H_n$ by a quantity depending only on $\theta$, for
 $G_n := \sum\limits_{j=1}^n a_{n,j}$, $H_n := \sum\limits_{j=1}^n W_j$. Such a bound is a consequence of \cite[Lemma 4.8]{ben2015permutation}, in the case where 
 $u_j = 1$ for $1 \leq j \leq n$. 
\end{proof}

The lemma proven here allows to compare the quantity $X_{\sigma_n, \delta_n} (f) $ with a linear combination of independent Poisson random variables, for which classical tools in probability theory can be used to prove limit theorems.  

\section{Proof of Theorem~\ref{thm:main} (i)}
We couple the variables $(a_{n, \ell})_{1 \leq \ell \leq n}$ with independent Poisson variables $(W_{\ell})_{\ell \geq 1} $ by using the Feller coupling, as in the previous section. 
From \eqref{Xxi}, we get
\begin{equation}
X_{\sigma_n, \delta_n}(f)
   = n \delta_n \hat{f}(0) + f(0) \sum_{\ell < \delta_n^{-1}} W_{\ell}   + f(0) \sum_{\ell < \delta_n^{-1}} (a_{n, \ell } - W_{\ell}) +  \sum_{\ell = 1}^n a_{n, \ell} \Xi_f\left( \frac{1}{\ell \delta_n}\right). \label{Xxi2}
   \end{equation}
In order to prove the first part of Theorem~\ref{thm:main}, we will show that the sum of the two first terms satisfies the same central limit theorem, 
and that the two  last term are bounded in $L^2$. We first prove the following result: 

\begin{proposition}\label{prop:gaussian}
We have: 
$$\mathbb{E} \left[ n \delta_n \hat{f}(0) + f(0) \sum_{\ell < \delta_n^{-1}} W_{ \ell}  \right] = n \delta_n \hat{f}(0)  - \theta \log (\delta_n) f(0) + 
\mathcal{O}_{f,\theta}(1),$$
$$\operatorname{Var} \left[ n \delta_n \hat{f}(0) + f(0) \sum_{\ell < \delta_n^{-1}} W_{\ell} \right] =   - \theta \log (\delta_n) f(0)^2 + 
\mathcal{O}_{f,\theta}(1),$$
and 
$$\frac{n \delta_n \hat{f}(0) + f(0) \sum_{\ell < \delta_n^{-1}} W_{\ell}  - 
\mathbb{E} \left[ n \delta_n \hat{f}(0) + f(0) \sum_{\ell < \delta_n^{-1}} W_{\ell}  \right] }{ \sqrt{ \operatorname{Var} \left[ n \delta_n \hat{f}(0) + f(0) \sum_{\ell < \delta_n^{-1}} W_{\ell}  \right]}} \overset{d}{\underset{n \rightarrow \infty}{\longrightarrow}} \mathcal{N}(0,1).$$
\end{proposition}

\begin{proof}
Since $(W_{\ell})_{\ell \geq 1}$ are independent Poisson random variables, $W_{\ell}$ with parameter $\theta/\ell$, we get
\begin{equation}\label{eq:espvarW}
\mathbb{E} \left( \sum_{\ell < \delta_n^{-1}} W_\ell \right) = \mathrm{Var} \left( \sum_{\ell < \delta_n^{-1}} W_\ell \right) = \sum_{\ell < \delta_n^{-1}} \frac{\theta}{\ell} =  \theta \log (\delta_n^{-1}) + \mathcal{O}_\theta (1),
\end{equation}
which gives the estimates of the proposition for the expectation and the variance. The central limit theorem is easily obtained by applying 
the Lindeberg-Feller criterion, since the variables $(W_{\ell})_{\ell \geq 1}$ are independent. 
\end{proof}
We then prove that the two last terms of \eqref{Xxi2} are bounded in $L^2$: 
\begin{proposition}
We have the estimate: 
$$ \mathbb{E} \left[ \left( f(0) \sum_{\ell < \delta_n^{-1}} |a_{n, \ell } - W_{\ell}| +  \sum_{\ell = 1}^n a_{n, \ell} \left| \Xi_f\left( \frac{1}{\ell \delta_n}\right)  \right| \right)^2 \right] 
= \mathcal{O}_{f, \theta}(1)$$
\end{proposition}
\begin{proof}
By Lemma \ref{lem:fellercoupling}, it is enough to show 
$$ \mathbb{E} \left[ \left( \sum_{\ell = 1}^n a_{n, \ell} \left| \Xi_f\left( \frac{1}{\ell \delta_n}\right)  \right| \right)^2 \right] 
= \mathcal{O}_{f, \theta}(1).$$
Moreover, we have $a_{n, \ell} \leq W_{\ell}$ for all $\ell$ except at most one value, for which we may have $a_{n, \ell}  = W_{\ell} + 1$. 
It is then enough to check
$$ \mathbb{E} \left[ \left(  \sup_{\mathbb{R}_+^*} |\Xi_f| + \sum_{\ell = 1}^n W_{\ell} \left| \Xi_f\left( \frac{1}{\ell \delta_n}\right)  \right| \right)^2 \right] 
= \mathcal{O}_{f, \theta}(1),$$
or equivalently
$$\mathbb{E} \left[   \sup_{\mathbb{R}_+^*} |\Xi_f| + \sum_{\ell = 1}^n W_{\ell} \left| \Xi_f\left( \frac{1}{\ell \delta_n}\right)  \right|  \right] 
= \mathcal{O}_{f, \theta}(1)$$
 and 
 $$\operatorname{Var} \left( \sum_{\ell = 1}^n W_{\ell} \left| \Xi_f\left( \frac{1}{\ell \delta_n} \right) \right| \right) = \mathcal{O}_{f, \theta}(1).$$
These estimates are implied by the estimate 
$$ \sup_{\mathbb{R}_+^*} |\Xi_f| + \sum_{\ell = 1}^n \frac{1}{\ell} \left( \left| \Xi_f\left( \frac{1}{\ell \delta_n} \right) \right| + \left| \Xi_f\left( \frac{1}{\ell \delta_n} \right) \right|^2 \right)
= \mathcal{O}_{f}(1),$$
which is a direct consequence of Lemma \ref{lem:xi}. 
\end{proof} 
It is now easy to deduce Theorem~\ref{thm:main} (i) from the two propositions just above. The estimate of the expectation is immediate, and the estimate of the variance is directly deduced from the fact that 
$$\operatorname{Var} (A+B) = \operatorname{Var} (A) +  \operatorname{Var} (B)  + \mathcal{O} \left( \operatorname{Var}^{1/2} (A)  \operatorname{Var}^{1/2} (B) \right).$$
For the central limit theorem, we know that
$$\frac{n \delta_n \hat{f}(0) + f(0) \sum_{\ell < \delta_n^{-1}} W_{\ell}  - 
\mathbb{E} \left[ n \delta_n \hat{f}(0) + f(0) \sum_{\ell < \delta_n^{-1}} W_{\ell}  \right] }{ \sqrt{ \operatorname{Var} \left[ n \delta_n \hat{f}(0) + f(0) \sum_{\ell < \delta_n^{-1}} W_{\ell}  \right]}} \overset{d}{\underset{n \rightarrow \infty}{\longrightarrow}} \mathcal{N}(0,1).$$
and 
$$\frac{f(0) \sum_{\ell < \delta_n^{-1}} (a_{n, \ell } - W_{\ell}) +  \sum_{\ell = 1}^n a_{n, \ell}  \Xi_f\left( \frac{1}{\ell \delta_n}\right)}{ \sqrt{ \operatorname{Var} \left[ n \delta_n \hat{f}(0) + f(0) \sum_{\ell < \delta_n^{-1}} W_{\ell}  \right]}} {\underset{n \rightarrow \infty}{\longrightarrow}}  0$$
in $L^2$, since the numerator is bounded in $L^2$ and the denominator tends to infinity with $n$. Hence, 
$$\frac{\mathbb{E} \left[ f(0) \sum_{\ell < \delta_n^{-1}} (a_{n, \ell } - W_{\ell}) +  \sum_{\ell = 1}^n a_{n, \ell}  \Xi_f\left( \frac{1}{\ell \delta_n}\right) \right]}{ \sqrt{ \operatorname{Var} \left[ n \delta_n \hat{f}(0) + f(0) \sum_{\ell < \delta_n^{-1}} W_{\ell}  \right]}} {\underset{n \rightarrow \infty}{\longrightarrow}}  0,$$
and by Slutsky's lemma,
$$\frac{ X_{\sigma_n, \delta_n} - \mathbb{E} [ X_{\sigma_n, \delta_n} ] }{ \sqrt{ \operatorname{Var} \left[ n \delta_n \hat{f}(0) + f(0) \sum_{\ell < \delta_n^{-1}} W_{\ell}  \right]}} \overset{d}{\underset{n \rightarrow \infty}{\longrightarrow}} \mathcal{N}(0,1).$$
Since the variance estimates we know imply that 
$$\frac{ \sqrt{ \operatorname{Var} \left( n \delta_n \hat{f}(0) + f(0) \sum_{\ell < \delta_n^{-1}} W_{\ell}  \right)}}{ \sqrt{ \operatorname{Var} (X_{\sigma_n, \delta_n})}}
\underset{n \rightarrow \infty}{\longrightarrow} 1,$$
we are done.

\section{Proof of Theorem~\ref{thm:main} (ii)}

Let $A_n:= \sum_{\ell = 1}^n a_{n,\ell} \Xi_f \left( \frac{1}{\ell \delta_n}\right)$, $B_n:= \sum_{\ell = 1}^n W_\ell \Xi_f \left( \frac{1}{\ell \delta_n}\right)$ and $Z=\sum_{y\in \mathcal{X}}  \Xi_f (y)$. Here, $a_{n,\ell}$ and $W_{\ell}$ are again related by the Feller coupling.  Notice that the sum defining $Z$ is a.s. absolutely convergent, since 
$$\mathbb{E} \left[ \sum_{y\in \mathcal{X}}  |\Xi_f (y)| \right]  = \theta \int_0^{\infty}  |\Xi_f (x)|  \frac{dx}{x} < \infty$$
by Lemma \ref{lem:xi}. 

 We are going to prove the result in two steps:
\begin{enumerate}
	\item[(i)] For all $t\in \mathbb{R}$, $\mathbb{E}(\mathrm{e}^{it B_n}) \underset{n \to \infty}{\longrightarrow} \mathbb{E}(\mathrm{e}^{it Z})$.
	\item[(ii)] For all $t\in \mathbb{R}$, $\left\vert \mathbb{E}(\mathrm{e}^{it A_n}) - \mathbb{E}(\mathrm{e}^{it B_n})  \right\vert \underset{n \to \infty}{\longrightarrow} 0$.
\end{enumerate}
(i): Let $t\in \mathbb{R}$. Using that the variables $W_j$ are independent,
\begin{align*}
\mathbb{E}(\mathrm{e}^{it B_n})  &= \prod_{\ell = 1}^n \mathbb{E} \left( \mathrm{e}^{it W_\ell \Xi_f (1/\ell\delta_n)}  \right) \\
	&= \prod_{\ell =1}^n \exp \left( \frac{\theta}{\ell} \left(  \mathrm{e}^{it \Xi_f (1/\ell\delta_n)} -1 \right)  \right) \\
	&=  \exp \left( \theta \delta_n \sum_{\ell = 1}^n \frac{1}{\ell\delta_n} \left(  \mathrm{e}^{it \Xi_f (1/\ell\delta_n)} -1 \right)  \right) \\
	&=  \exp \left( \theta \delta_n \sum_{1 \leq \ell \leq n,  \, \ell \delta_n \in [1/R, R] } \frac{1}{\ell\delta_n} \left(  \mathrm{e}^{it \Xi_f (1/\ell\delta_n)} -1 \right)  \right) 
	\\ & \times \exp \left( \theta \delta_n \sum_{1 \leq \ell \leq n,  \, \ell \delta_n \notin [1/R, R] } \frac{1}{\ell\delta_n} \left(  \mathrm{e}^{it \Xi_f (1/\ell\delta_n)} -1 \right)  \right), 
\end{align*}	
for any $R > 1$. For fixed $R$ and $n$ large enough depending on $R$, the condition $1 \leq \ell \leq n$ can be discarded in the first exponential of the last  product, since  $\delta_n \rightarrow 0$ and $n \delta_n \rightarrow \infty$ when $n \rightarrow \infty$. The sum in the first exponential is then a Riemann sum, which by continuity of $\Xi_f$ (proven in Lemma \ref{lem:xi}: recall that $f(0) = 0$ in this section), shows that the exponential tends to 
	$$ \exp \left( \theta \int_{1/R}^{R}  \frac{1}{x} \left(  \mathrm{e}^{it \Xi_f (1/x)} -1 \right) \mathrm{d}x \right)$$
when $n$ goes to infinity. On the other hand, by Lemma \ref{lem:xi} the sum inside the second exponential is dominated by
$$\sum_{\ell \delta_n > R}  \frac{1}{(\ell\delta_n)^2} + \sum_{\ell \delta_n < 1/R} 1 = \mathcal{O} ((R \delta_n)^{-1}),$$
and  then the second exponential is
 $$\exp \left( \mathcal{O}_{f, \theta, t} (1/R) \right) = 1 + \mathcal{O}_{f, \theta, t} (1/R).$$
 Now, if $L$ is the limit of a subsequence $(\mathbb{E}(\mathrm{e}^{it B_{n_k}}))_{k \geq 1}$, then 
\begin{align*}
&  \exp \left( \theta \delta_{n_k} \sum_{1 \leq \ell \leq n_k,  \, \ell \delta_{n_k} \notin [1/R, R] } \frac{1}{\ell\delta_{n_k}} \left(  \mathrm{e}^{it \Xi_f (1/\ell\delta_{n_k})} -1 \right)  \right)
\\ & \underset{k \rightarrow \infty}{\longrightarrow} L  \exp \left( -  \theta \int_{1/R}^{R}  \frac{1}{x} \left(  \mathrm{e}^{it \Xi_f (1/x)} -1 \right) \mathrm{d}x \right).
\end{align*}
Since the left-hand side of the convergence is $1 + \mathcal{O}_{f, \theta, t} (1/R)$, we deduce that 
$$L =  ( 1 + \mathcal{O}_{f, \theta, t} (1/R) ) \exp \left(    \theta \int_{1/R}^{R}  \frac{1}{x} \left(  \mathrm{e}^{it \Xi_f (1/x)} -1 \right) \mathrm{d}x \right).$$
Letting $R \rightarrow \infty$, we deduce 
$$L = \exp \left(    \theta \int_{0}^{\infty}  \frac{1}{x} \left(  \mathrm{e}^{it \Xi_f (1/x)} -1 \right) \mathrm{d}x \right)
= \exp \left(    \theta \int_{0}^{\infty}  \frac{1}{y} \left(  \mathrm{e}^{it \Xi_f (y)} -1 \right) \mathrm{d}y \right),$$
where the convergence of the integrals is insured by the integrability of $|\Xi_f(x)| dx/x$ given in Lemma \ref{lem:xi}. 
By Campbell's theorem, 
$$ L = \mathbb{E}(\mathrm{e}^{it Z}),$$
i.e. $\mathbb{E}(\mathrm{e}^{it Z})$ is the unique possible limit of a subsequence of $(\mathbb{E}(\mathrm{e}^{it B_{n}}))_{n \geq 1}$. Since this sequence is bounded, we have proven (i). 

(ii): Let $t\in \mathbb{R}$.
\begin{align*}
\left\vert \mathbb{E} (\mathrm{e}^{itA_n}) - \mathbb{E}(\mathrm{e}^{itB_n}) \right\vert &\leq \vert t \vert \mathbb{E} ( \vert A_n - B_n \vert ) \\
	&\leq \vert t \vert \sum_{\ell =1}^n \vert \Xi_f (1/ \ell \delta_n) \vert \mathbb{E} ( \vert a_{n,\ell } - W_\ell \vert  )
\end{align*}
where, by \cite[Lemma 4.4]{ben2015permutation}, 
\[\mathbb{E} ( \vert a_{n,\ell } - W_\ell \vert ) \leq \frac{C(\theta )}{n} + \frac{\theta}{n} \Psi_n (\ell ) \ll_\theta \frac{1}{n} (1 + \Psi_n (\ell))\]
for some $C(\theta) > 0$ depending only on $\theta$ and for 
$$\Psi_n (\ell) := \prod_{k=0}^{\ell -1} \frac{n-k}{ \theta + n - k - 1}.$$
Let $(u_n)_{n \geq 1}$ and $(v_n)_{n \geq 1}$ be two sequences of positive integers such that
$u_n$, $\delta_n^{-1}/ u_n$, $v_n / (\delta_n^{-1})$ and $n/v_n$ all go to infinity with $n$. 
 On the one hand, 
\begin{align*}
	&\frac{1}{n} \sum_{\substack{\ell < u_n \\ \text{or } \ell >v_n}} (1 + \Psi_n (\ell)) \vert \Xi_f (1/ \ell \delta_n) \vert \\
	&\qquad \leq \left( \sup_{z \in (0, u_n \delta_n) \cup (v_n \delta_n, +\infty)} \vert \Xi_f (1/z) \vert \right) \frac{1}{n} \sum_{\ell =1 }^n (1 + \Psi_n (\ell ))  \\
	&\qquad = \left( 1 + \frac{1}{\theta} \right) \sup_{z \in (0, u_n \delta_n) \cup (v_n \delta_n, +\infty)} \vert \Xi_f (1/z) \vert \\
	&\qquad \underset{n\to\infty}{\longrightarrow} 0,
\end{align*}
because $ \Xi_f $ tends to zero at zero and at infinity, and by  \cite[Lemma 4.6]{ben2015permutation}, 
$$\frac{\theta}{n} \sum_{\ell =1}^n  \Psi_n (\ell ) = \sum_{\ell =1}^n  \mathbb{P} [ J_n = \ell] = 1$$
where 
$$J_n = \min \{j \geq 1, \xi_{n-j+1} = 1 \},$$
$(\xi_j)_{j \geq 1}$ being independent Bernoulli variables, $\xi_j$ having parameter $\theta/ (\theta + j -1)$. 
On the other hand, since $\Psi_n (\ell)$ is monotonic with respect to $\ell$ and tends to  $1$ when $n/\ell$ goes to infinity. 
\begin{align*}
	&\frac{1}{n} \sum_{u_n \leq \ell \leq v_n }  (1 + \Psi_n (\ell)) \vert \Xi_f (1/ \ell \delta_n) \vert \\
	&\qquad \leq \Vert \Xi_f \Vert_\infty \frac{v_n - u_n +1 }{n} \times (1 + \max (\Psi_n (u_n) , \Psi_n (v_n) )) \\
	&\qquad \underset{n\to\infty}{\longrightarrow} 0
\end{align*}
since $n/v_n \to \infty$ and then $\max (\Psi_n (u_n) , \Psi_n (v_n) ) \to 1$ as $n \to \infty$.

\section{Proof of Theorem~\ref{thm:main} (iii) and related statements}

Since $M^{\sigma_n}$ has $n$ eigenangles in each interval of length $2 \pi$, replacing $X_{\sigma_n, \delta_n}$ by $X'_{\sigma_n, \delta_n}$  changes the sum by at most
$$n \sum_{k \neq 0} \sup_{|x - 2 k \pi | \leq \pi} |f(x/2 \pi \delta_n)| \ll_f n \sum_{k \geq 1} \delta_n^{\alpha} k^{-\alpha} \ll_{f, \alpha} n \delta_n^{\alpha},$$
quantity which, by the assumption made in (iii),  tends to zero when $n \rightarrow \infty$. 
Using Slutsky's lemma, we easily deduce that (i) and (ii) are preserved when we replace $X_{\sigma_n, \delta_n}$ by $X'_{\sigma_n, \delta_n}$. 
In the sequel of this section, we will state alternative assumptions under which (ii) is preserved.
Let us first show  the following lemma:
\begin{lemma}
For all $n\geq 1$, 
\begin{equation}\label{eq:psi1}
\frac{1}{n}\sum_{j=1}^n \Psi_n (j) = \frac{1}{\theta},
\end{equation}
\begin{equation}\label{eq:psi2}
\sum_{j=1}^n \frac{\Psi_n (j)}{j} = \sum_{j=1}^n \frac{1}{\theta + j-1},
\end{equation}
and
\begin{equation}\label{eq:psi3}
\sum_{j=1}^n \frac{\Psi_n (j)}{j^2} \underset{n\to\infty}{\longrightarrow} \frac{\pi^2}{6}.
\end{equation}
\end{lemma}

\begin{proof}
The equalities~\eqref{eq:psi1} and  \eqref{eq:psi2} are proven in \cite[Lemma 9]{bahier2019meso}. Let us now show \eqref{eq:psi3}. Let $(u_n)_{n \geq 1}$ be a sequence of positive integers such that $u_n$ and $n/u_n$ both tend to infinity with $n$. 
We split the sum into two as follows
\[\sum_{j=1}^n \frac{\Psi_n (j)}{j^2} = \sum_{j=1}^{u_n} \frac{\Psi_n (j)}{j^2} + \sum_{j=u_n + 1}^n \frac{\Psi_n (j)}{j^2}  \]
By monotonicity of $\Psi_n (k)$ with respect to $k$, we have 
\[\min\limits_{k \leq u_n} \Psi_n (k) = \min (\Psi_n (1) , \Psi_n (u_n)) \]
and 
\[\max\limits_{k \leq u_n} \Psi_n (k) = \max (\Psi_n (1) , \Psi_n (u_n)), \]
where $\lim\limits_{n \rightarrow \infty} \Psi_n (1) = \lim\limits_{n \rightarrow \infty} \Psi_n (u_n) = 1$. Thus, as $n$ goes to infinity, we have
\[\sum_{j=1}^{u_n} \frac{\Psi_n (j)}{j^2} = (1+ o (1)) \sum_{j=1}^{u_n} \frac{1}{j^2} = \frac{\pi^2}{6} + o (1).\]
Besides,
\begin{align*}
\sum_{j=u_n + 1}^n \frac{\Psi_n (j)}{j^2} &\leq \max\limits_{k = u_n+1, \dots ,n} \Psi_n (k) \times \sum_{j=u_n + 1}^{+\infty} \frac{1}{j^2} \\
	&\ll_\theta \max(1, n^{1-\theta}) \times \frac{1}{u_n}
\end{align*}
which tends to $0$ as $n$ goes to infinity if we take for instance $u_n := \lfloor \max (n^{1- \frac{\theta}{2}}, n^{1/2}) \rfloor$.
\end{proof}

Now, let us introduce the following notations: for all positive integers $j$ and all real numbers $x>0$, 
\[\Theta_{f,j} (x) := \sum_{k=\lfloor -j/2 \rfloor + 1}^{\lfloor j/2 \rfloor} f(kx),\]
and
\[\Xi_{f,j} (x) := \Theta_{f,j} (x) - f(0) \mathds{1}_{x>1} - \frac{1}{x} \hat{f}(0).\] 

With this notation, all computations related to  $X'_{\sigma_n, \delta_n}$ are similar to the computations related to $X_{\sigma_n, \delta_n}$, except 
that $\Theta_{f} $ and $\Xi_{f} $ are replaced by $\Theta_{f,\ell} $ and $\Xi_{f,\ell} $ in the contribution of a cycle of length $\ell$. 
We will now prove the following result: 
\begin{theorem}\label{thm:main2}
Assume \eqref{eq:hyp}, $f(0) = 0$ and 
\begin{equation}\label{eq:hyp2}
 x \int_{\vert u \vert > x} \vert f^{\prime \prime} (u) \vert  \mathrm{d}u \underset{x\to +\infty}{\longrightarrow} 0.
\end{equation}
If the sequence  $(\delta_n)_{n \geq 1}$ is such that
\[n\delta_n \int_{\vert u \vert > \frac{1}{\delta_n}} f(u) \mathrm{d}u \underset{n\to +\infty}{\longrightarrow} 0 \]
and 
\[\log (n) f \left( \pm \frac{1}{2\delta_n} \right) \underset{n\to +\infty}{\longrightarrow} 0,\]
then
\[ X_{\sigma_n, \delta_n} (f) - n \delta_n \hat{f}(0) = \sum_{\ell = 1}^n a_{n,\ell} \Xi_{f,\ell} \left( \frac{1}{\ell \delta_n}\right) \overset{d}{\underset{n\to\infty}{\longrightarrow}} \sum_{y\in \mathcal{X}}  \Xi_f (y)  \]
where $\mathcal{X}$ is a Poisson point process with intensity $\frac{\theta}{x} \mathrm{d}x$ on $(0,+\infty)$.
\end{theorem}

\begin{proof}
From Theorem~\ref{thm:main}, we are done if we show
\[\left\vert \mathbb{E}(\mathrm{e}^{it A_n}) - \mathbb{E}(\mathrm{e}^{it C_n})  \right\vert \underset{n \to \infty}{\longrightarrow} 0\]
for all $t\in \mathbb{R}$, where $A_n:= \sum_{\ell = 1}^n a_{n,\ell} \Xi_f \left( \frac{1}{\ell \delta_n}\right)$ and $C_n:= \sum_{\ell = 1}^n a_{n,\ell} \Xi_{f,\ell} \left( \frac{1}{\ell \delta_n}\right)$.
Let $t\in \mathbb{R}$.
\begin{align*}
\left\vert \mathbb{E}(\mathrm{e}^{it A_n}) - \mathbb{E}(\mathrm{e}^{it C_n})  \right\vert &\leq \vert t \vert \mathbb{E} (\vert A_n - C_n \vert )\\
	&\leq \vert t \vert \sum_{\ell = 1}^n \mathbb{E} (a_{n,\ell}) \left\vert \Xi_f \left( \frac{1}{\ell \delta_n}\right) - \Xi_{f,\ell} \left( \frac{1}{\ell \delta_n}\right) \right\vert \\
	&= \vert t \vert\sum_{\ell = 1}^n \frac{\theta \Psi_n (\ell)}{\ell} \left\vert \sum_{k=-\infty}^{\lfloor -\ell /2\rfloor} f \left( \frac{k}{\ell \delta_n} \right) + \sum_{k=\lfloor \ell /2\rfloor + 1}^{+\infty} f \left( \frac{k}{\ell \delta_n} \right)\right\vert 
\end{align*}
Here, we use the fact that the expectation of the number of $\ell$-cycles is equal to $n/\ell$ times the probability that $1$ is in an $\ell$-cycle, i.e., 
by the Feller coupling, 
$$\mathbb{E} [ a_{n,\ell} ] = \frac{n}{\ell} \mathbb{P} [ \xi_n = \xi_{n-1} = \dots = \xi_{n+2 - \ell} = 0, \xi_{n+1 - \ell} = 1]
$$ $$= \frac{n}{\ell} \frac{\theta}{n-\ell + \theta} \prod_{k= 1}^{\ell-1} \frac{n-k}{n-k + \theta} = \frac{\theta \Psi_n (\ell)}{\ell}.
$$
We now estimate the sum over the positive indices $\sum_{k=\lfloor \ell /2\rfloor + 1}^{+\infty} f \left( \frac{k}{\ell \delta_n} \right)$: the sum over the negative indices behaves identically.
To do this, we use the Euler-MacLaurin formula at order $2$: for all positive integers $p<q$, and for all functions $g \in \mathcal{C}^2 (\mathbb{R})$, 
\[\sum_{k=p}^q g(k) = \int_p^q g(x) \mathrm{d}x + \frac{g(p) + g(q)}{2} + \frac{g^\prime (q) - g^\prime (p)}{12} + \mathcal{O} \left( \int_p^q \vert g^{\prime \prime} (x) \vert \mathrm{d}x \right)\]
so that if $g$, $g^\prime$ and $g^{\prime \prime}$ are integrable at $+\infty$, we have, letting $q$ tend to infinity,
\[\sum_{k=p}^{+\infty} g(k) = \int_p^{+\infty} g(x) \mathrm{d}x + \frac{g(p)}{2} - \frac{g^\prime (p)}{12} + \mathcal{O} \left( \int_p^{+\infty} \vert g^{\prime \prime} (x) \vert \mathrm{d}x \right)\]  
Applying this formula to $g(x) = f\left( \frac{x}{\ell \delta_n}\right)$ gives, with a change of variables into the integrals,
\[\sum_{k=p}^{+\infty} f\left( \frac{k}{\ell \delta_n} \right) = -\ell \delta_n F \left( \frac{p}{\ell\delta_n} \right) + \frac{f\left( \frac{p}{\ell \delta_n} \right)}{2} - \frac{f^\prime \left( \frac{p}{\ell \delta_n} \right)}{12 \ell \delta_n} + \mathcal{O} \left( \frac{1}{\ell\delta_n} \int_{\frac{p}{\ell\delta_n}}^{+\infty} \vert f^{\prime \prime} (u) \vert \mathrm{d}u \right)\] 
where $F$ is the antiderivative of $f$ such that $F(+\infty ) =0$. \\
Then, with $p=\lfloor \ell /2\rfloor + 1$, using Taylor-Lagrange formula at order $3$ on $F$, at order $2$ on $f$ and at order $1$ on $f^\prime$, between $\frac{1}{2\delta_n}$ and $\frac{p}{\ell \delta_n} = \frac{1}{2\delta_n} + \frac{1-\{\ell/2\}}{\ell\delta_n}$, we get
\begin{align*}
\sum_{k=\lfloor \ell /2\rfloor + 1}^{+\infty} f\left( \frac{k}{\ell \delta_n} \right) &= -\ell \delta_n F \left( \frac{1}{2\delta_n} \right) + \left[ - (1-\{\ell/2 \}) + \frac{1}{2} \right] f \left( \frac{1}{2\delta_n} \right) \\
	&\quad +\frac{1}{\ell\delta_n} \left[ - \frac{1}{2} (1-\{\ell/2 \})^2 + \frac{1}{2}(1-\{\ell/2 \}) - \frac{1}{12}\right] f^\prime \left( \frac{1}{2\delta_n} \right) \\
	&\quad +  \mathcal{O} \left( \frac{1}{\ell\delta_n} \int_{\frac{1}{2\delta_n}}^{+\infty} \vert f^{\prime \prime} (u) \vert \mathrm{d}u \right).
\end{align*}
Consequently, 
\begin{align*}
\left\vert \sum_{k=\lfloor \ell /2\rfloor + 1}^{+\infty} f\left( \frac{k}{\ell \delta_n} \right) \right\vert &\leq \ell \delta_n \left\vert F \left( \frac{1}{2\delta_n} \right) \right\vert + \frac{1}{2} \left\vert f \left( \frac{1}{2\delta_n} \right) \right\vert \\
	&\quad + \frac{1}{12\ell\delta_n}  \left\vert f^\prime \left( \frac{1}{2\delta_n} \right) \right\vert \\
	&\quad + \mathcal{O} \left( \frac{1}{\ell\delta_n} \int_{\frac{1}{2\delta_n}}^{+\infty} \vert f^{\prime \prime} (u) \vert \mathrm{d}u \right).
\end{align*}
Finally, using \eqref{eq:psi1}, \eqref{eq:psi2} and \eqref{eq:psi3}, it follows
\begin{align*}
\sum_{\ell = 1}^n \frac{\theta \Psi_n (\ell)}{\ell} \left\vert \sum_{k=\lfloor \ell /2\rfloor + 1}^{+\infty} f\left( \frac{k}{\ell \delta_n} \right) \right\vert &\leq n \delta_n \left\vert F \left( \frac{1}{2\delta_n} \right) \right\vert  \\
	&\quad + \left(\frac{\theta}{2}\log n + \mathcal{O}_\theta (1)\right) \left\vert f \left( \frac{1}{2\delta_n} \right) \right\vert \\
	&\quad + \left( \frac{\theta \pi^2}{72} + o_\theta (1) \right) \frac{1}{\delta_n} \left\vert f^\prime \left( \frac{1}{2\delta_n} \right) \right\vert \\
	&\quad + \mathcal{O}_{\theta} \left( \frac{1}{\delta_n} \int_{\frac{1}{2\delta_n}}^{+\infty} \vert f^{\prime \prime} (u) \vert \mathrm{d}u \right)
\end{align*}
which tends to $0$ as $n\to +\infty$, under the hypothesis made on $f$ and $\delta_n$.
\end{proof}

\begin{example}
\begin{itemize}
\item If $f \in \mathcal{C}^2_c (\mathbb{R})$ (\emph{i.e.} $\mathcal{C}^2$ and compactly supported on $\mathbb{R}$), then all the conditions of Theorem~\ref{thm:main2} are satisfied, and this for every $\delta_n$.
\item If $f$ satisfies \eqref{eq:hyp}, \eqref{eq:hyp2} and if $n\delta_n^\alpha \to 0$ (it is in particular the case if $\delta_n = n^{-\varepsilon}$ for any $\varepsilon \in \left( \frac{1}{\alpha}, 1 \right)$), then all the conditions of Theorem~\ref{thm:main2} are satisfied. Indeed,
\[ \left\vert n\delta_n \int_{\vert u \vert > \frac{1}{\delta_n}} f(u) \mathrm{d}u \right\vert \leq n \delta_n \int_{\vert u \vert > \frac{1}{\delta_n}} \frac{M}{(1+ \vert u \vert)^\alpha} \mathrm{d}u \ll n \delta_n^\alpha \]
and
\[\left\vert \log (n) f\left( \pm \frac{1}{2\delta_n}  \right) \right\vert \ll \log (n) \delta_n^\alpha = o(n\delta_n^\alpha).\]
\item If $f \in \mathcal{S} (\mathbb{R})$ (\emph{i.e.} in the Schwartz space of $\mathbb{R}$), and if $\delta_n = n^{-\varepsilon}$ for any $\varepsilon \in \left( 0 , 1 \right)$, then all the conditions of Theorem~\ref{thm:main2} are satisfied.
\end{itemize}
\end{example}

\subsubsection*{Some counterexamples}
\begin{itemize}
 \item If $f(x) = 1/(1+|x|)$, then $X_{\sigma_n, \delta_n} (f)$ is infinite, whereas in the expression of $X'_{\sigma_n, \delta_n} (f)$, a cycle of length $\ell$ 
gives a contribution of 
\begin{align*}
& \sum_{k = \lfloor - \ell/2  \rfloor + 1}^{\lfloor  \ell/2 \rfloor} \frac{1}{ 1 + |k|/(\ell \delta_n)} = \mathcal{O}(1)  + 2 \sum_{k = 1}^{\lfloor \ell/2 \rfloor} \frac{\ell \delta_n}{\ell \delta_n + k}
 \\ & = 2 \ell \delta_n \left( \log \left( \frac{ \ell \delta_n + \ell/2}{\ell \delta_n + 1} \right) + \mathcal{O}(1) \right) + \mathcal{O}(1)  
 \\ & = 2 \ell \delta_n ( \log \ell -  \log ( 1+ \ell \delta_n)  + \mathcal{O}(1) ) +  \mathcal{O}(1).
 \end{align*}
 If $\ell \geq \delta_n^{-1}$, we get an estimate: 
 $$2 \ell \delta_n ( \log \ell -  \log ( \ell \delta_n)  + \mathcal{O}(1) ) +  \mathcal{O}(1) = 2 \ell \delta_n \log (\delta_n^{-1} ) + \mathcal{O} (\ell \delta_n),$$
 and if $\ell \leq \delta_n^{-1}$, we get
 $$2 \ell \delta_n ( \log \ell  + \mathcal{O}(1) ) +  \mathcal{O}(1)  = 2 \ell \delta_n  \log \ell  + \mathcal{O}(1)$$
The sum of the lengths of the cycles larger than $\delta_n^{-1}$ is $n- o(n)$ with probability tending to $1$ when $n \rightarrow \infty$, and their contribution 
is then equivalent to $2 n \delta_n  \log (\delta_n^{-1} )$.
The contribution of the cycles smaller than $\delta_n^{-1}$ is dominated by $\delta_n \log (\delta_n^{-1})$ times the sum of their lengths, plus the number of these cycles. 
Using the Feller coupling with independent Poisson variables, one deduces that with high probability, the contribution is dominated by 
$$\delta_n \log (\delta_n^{-1}) (\delta_n^{-1} \omega(n)) + \log (\delta_n^{-1}) \ll \log (\delta_n^{-1}) \omega(n)$$
for any function $\omega (n)$ larger than $1$ and going to infinity at infinity. 
If we take $\omega (n)$ going to infinity slower than $n \delta_n$, we deduce that 
$$X'_{\sigma_n, \delta_n} (f) = (2 + o(1)) n \delta_n  \log (\delta_n^{-1} )$$
with probability tending to $1$ when $n \rightarrow \infty$. 
This behavior at infinity does not correspond to what we get in the theorems proven earlier. 
\item If $f(x) = x^2/(1+x^4)$, $f$ satisfies the assumptions of Theorem~\ref{thm:main} (ii), and then 
$$X_{\sigma_n, \delta_n} - n \delta_n \hat{f}(0) \overset{d}{\underset{n \rightarrow \infty}{\longrightarrow}} \sum_{y \in \mathcal{X}} \Xi_f (y).$$
If we replace $X_{\sigma_n, \delta_n}$ by $X'_{\sigma_n, \delta_n}$, then we subtract at least $n$ terms of the form $f(x/2 \pi \delta_n)$ for 
$\pi < x \leq 3 \pi$, and then at least a quantity of order $n \delta_n^2$. If $n \delta_n^2$ tends to infinity when $n \rightarrow \infty$ (for example if 
$\delta_n = n^{-1/3}$), 
then 
$$X'_{\sigma_n, \delta_n} - n \delta_n \hat{f}(0) \overset{d}{\underset{n \rightarrow \infty}{\longrightarrow}} - \infty,$$
in the sense that 
$$\mathbb{P} [ X'_{\sigma_n, \delta_n} - n \delta_n \hat{f}(0) > -A] \underset{n \rightarrow \infty}{\longrightarrow} 0$$
for any fixed $A > 0$. Hence, Theorem~\ref{thm:main} (ii) does not extend to $X'_{\sigma_n, \delta_n} $. 
\end{itemize}
\newpage
\nocite{*}
\bibliographystyle{plain}
\bibliography{biblio}

\end{document}